\documentclass[12pt]{amsart}     

 \usepackage[applemac]{inputenc}

 \usepackage{amsmath}     

 \usepackage{amssymb}    

 \usepackage{euscript}

\newtheorem{thm}{Theorem}[]

\newtheorem{lem}[thm]{Lemma}

\newtheorem{cor}[thm]{Corollary}

\theoremstyle{definition}

\newtheorem{defn}[thm]{Definition}

\theoremstyle{remark}

\newtheorem{remark}[thm]{Remark}

\newtheorem{remarks}[thm]{Remarks}

{\em}

{\em}

\newcounter{substep}

\def\thesubstep{\arabic{substep}}

{\em}

\newcounter{subsubstep}

\def\thesubsubstep{\arabic{subsubstep}}

{\em}

\newcommand{\ideal}{{\mathcal I}}

\newcommand{\zeri}{{\mathcal Z}}

\newcommand{\gta}{{\mathfrak a}} \newcommand{\gtb}{{\mathfrak b}}

\newcommand{\gtg}{{\mathfrak g}}

\newcommand{\Fhaz}{{\EuScript F}}

\newcommand{\an}{{\EuScript O}}

\newcommand{\stb}{{\EuScript E}}

\newcommand{\supp}{\operatorname{supp}}

\newcommand{\Int}{\operatorname{Int}}

\begin{document}

\title[$C^\infty$ Nullstellensatz]{A Nullstellensatz for \L ojasiewicz ideals}

% author one information

\author{Francesca Acquistapace, Fabrizio Broglia,}

\address{Dipartimento di Matematica, Universit\`a degli Studi di Pisa, Largo Bruno Pontecorvo, 5, 56127 Pisa, Italy}

\email{acquistf@dm.unipi.it, broglia@dm.unipi.it}

\author{Andreea Nicoara}

\address{Department of Mathematics, University of Pennsylvania, 209 South $33^{rd}$ St.,  Philadelphia, PA 19104}

\email{anicoara@math.upenn.edu}

\thanks{}

% \date is required; it is the date received by the editor.

\date{October 9, 2012}



\begin{abstract}

For an  ideal of  smooth functions  $\gta$ that is  either \L  ojasiewicz or
weakly \L  ojasiewicz, we give a  complete characterization of  the ideal of
functions  vanishing on its  variety $\ideal(\zeri(\gta))$  in terms  of the
global \L ojasiewicz radical and Whitney  closure. We also prove that the \L
ojasiewicz radical of such an ideal is {\sl analytic-like} in the sense that
its saturation  equals its Whitney closure.  This allows us to  recover in a
different way  Nullstellensatz results due  to Bochnak and  Adkins-Leahy and
answer positively  a modification of  the Nullstellensatz conjecture  due to
Bochnak.

\end{abstract}

%\date{September 2 2012}

\maketitle

\tableofcontents

\section{Introduction}\label{intro}

In  this paper  we  characterize a  class  of ideals  $\gta$  having the  zero
property in the algebra $\stb (M)$ of real-valued smooth functions on a smooth
manifold $M$.  Recall  that an ideal $\gta$ has the {\it  zero property} if it
coincides with  the ideal $\ideal(\zeri(\gta))$ of all  functions vanishing on
its zero set.

This  investigation of  such a  Nullstellensatz  for the  class of  $C^\infty$
functions was initiated by Bochnak  in 1973 in \cite{bochnak} and subsequently
continued by Risler in  \cite{risler}. Interesting contributions by Adkins and
Leahy can be found in \cite{al1} and \cite{al2}.

In particular, Bochnak  formulated the following conjecture:

\bf Conjecture.  \label{conj} \sl Let $\gta$  be a finitely  generated ideal in $\stb (M)$. Then the following are equivalent

\begin{enumerate}

\item $\gta$ has the zero property.

\item $\gta$ is closed and real.

\end{enumerate}

\rm

He proved  his conjecture when $\gta$  is generated by  finitely many analytic
functions. Then Risler in \cite{risler} gave a complete answer for dimension 2
and for  principal ideals in dimension  3. Finally, for an  ideal generated by
analytic functions (not necessarily finitely  many), Adkins and Leahy prove in
\cite{al1} that  $\ideal(\zeri(\gta))$ is the  closure of the real  radical of
$\gta$.

Note that a closed finitely  generated ideal is \L ojasiewicz, (see definition
below), but the converse is not true (\cite{tougeron} p.104 example 4.8).

In   Theorem   \ref{main},   we    give   a   complete   characterization   of
$\ideal(\zeri(\gta))$ for  the case  when $\gta$ is  a \L ojasiewicz  ideal in
terms of a particular notion of radical called the {\em \L ojasiewicz radical}
that can  be bigger than the real  radical. It is defined  below in Definition
\ref{L  radical}. \L  ojasiewicz  ideals were  considered  by several  authors
including  Malgrange  in   \cite[\S6]{malgrange},  Thom  in  \cite{thom},  and
Tougeron in \cite[p.104]{tougeron}. The  \L ojasiewicz radical appears in work
by Kohn  \cite[Thm.1.21]{kohnacta} and Nowak \cite{nowak}, though  mainly as a
notion applied to ideals of germs.

As  a consequence  of Theorem  \ref{main}, we  give an  answer to  the Bochnak
conjecture in terms of convexity:

\begin{thm}

Let $\gta$  be a \L ojasiewicz ideal  in $\stb (M)$. Then  the following are equivalent

\begin{enumerate}

\item $\gta$ has the zero property.

\item $\gta$ is closed, convex, and  radical.

\end{enumerate}

\end{thm}

In fact, we obtain our result for ideals $\gta$ with countably many generators
but  still  verifying  condition  (2)  of Definition  \ref{Loj}.  See  Theorem
\ref{weak L}.

Note  that  a  convex radical  ideal  is  a  real  ideal.  If we  had  a  good
representation  of positive  semidefinite functions  as sums  of  squares, the
converse would  also be true. Nevertheless,  this converse happens  to be true
when  $\gta$  is  generated by  analytic  functions  as  we prove  in  Theorem
\ref{comp}. We thus recover the results of Bochnak and Adkins-Leahy.

%nowak preprint and Kohn for the Lojasiewicz radical, and thom, malgrange book and tougeron book for Lojasiewicz ideals

%put the main theorems in this section, if possible.

\medskip

%\begin {center}

%*  *  *

%\end{center}

\section{\L ojasiewicz ideals}\label{loji}

\medskip

Let  $M$ be  a  smooth manifold  and  $ \stb  (M)$ be  its  algebra of  smooth
real-valued functions endowed with the compact open topology.

The {\sl saturation} of  an ideal $\gta$ in $ \stb (M)$  is the ideal $ \tilde
\gta = \{ g\in \stb (M) \, | \,\forall x \in M\ g_x \in \gta \stb_x \}.$

\begin{lem}\label{satclos}

$\gta\subset \tilde \gta \subset  \overline \gta$

\end{lem}

\begin{proof} Consider the ideal 
$$ \gta^* = \{ g\in \stb (M) \, | \,\forall x \in M\ T_x g \in  T_x\gta \}.$$
Whitney spectral theorem says $\gta^* = \overline{\gta}$, and the proof
follows since $\tilde \gta \subset \gta^*$.

\end{proof}

\begin{remarks}\label{alike}

\noindent

\begin{enumerate}

\item  Both inclusions  are strict  in  general; consult  Adkins-Leahy
\cite[p.708]{al1} for an example.

\item If  $M$ is  analytic and $g$  belongs to  the ring $\an(M)$  of analytic
functions on $M$,  then $g_x$ can be identified with  $T_x g$. A consequence
of this  fact is that for  any ideal $\gta \subset  \an(M),$ $ \gta^*=\tilde
\gta,$ where the two operations on $\gta$ are performed in the ring $\an(M)$
only. In this case $\gta^*$ is not the closure in the compact open topology;
see \cite{debbumi}.  From now on, we will call {\sl analytic-like} any ideal
in $\stb (M)$ verifying $ \gta^*=\tilde \gta.$

%\item  If $M$ is analytic and we consider ideals in the ring $\An(M)$ of analytic functions on $M,$ saturation is not closure in the compact open topology but in another suitable topology; see \cite{debbumi}.

\item  Note that

\begin{equation*}
\begin{split}
\tilde \gta &= \{g\in\stb (M) \, | \,\forall \,\text{ compact }  K\subset M \ \exists\,h\in \stb (M)\:\text{s.t.}\:\zeri(h)\cap K=\varnothing \: \text{and} \:hg\in\gta\} \\&= \{g\in \stb (M) \, | \,\forall\,x\in M\,\exists\,h\in \stb (M)\:\text{s.t.}\: h(x)\neq0 \: \text{and} \:hg\in\gta\}.
\end{split}
 \end{equation*}

 \begin{proof} It is clear that both  the second and third sets are subsets of
$ \tilde  \gta.$ It is also  clear that the second  set is a  subset of the
third. Therefore,  the three-way equality  reduces to proving  that $\tilde
\gta$ is  a subset  of the second  set. Let  $g\in \tilde \gta$.  Given any
compact subset $K$  of $M,$ let $x \in K.$  It follows $g_x\in \gta\stb_x$,
and in a suitable  neighborhood $U_x$ of $x,$ $$g = \alpha_1  f_1 + \dots +
\alpha_k  f_k.$$ Take  a  bump  function $\varphi$  such  that $x\in  \supp
\varphi \subset  U_x$ and $\varphi(y) \neq 0$  for any $y$ in a smaller
open set $ V_x.$ 

%Such a smooth   function   is   guaranteed   to   exist.   See   \cite[Lemma

%6.1 p.113]{tougeron}.  

Then $$\varphi  g = (\varphi\alpha_1)  f_1 + \dots +  (\varphi\alpha_k) f_k
\in \gta.$$  $\{V_x\}_{x \in K}$ is an  open cover of the  compact set $K.$
Take a  finite subcover $V_{x_1},  \dots, V_{x_j}$ with  corresponding bump
functions  $\varphi_1, \dots,  \varphi_j.$ Summing  up the  expressions for
$\varphi_1 g,  \dots, \varphi_j  g,$ we obtain  that $(\varphi_1+  \cdots +
\varphi_j)g$ is  a finite  sum of elements  of $\gta$ with  coefficients in
$\stb (M).$ As $(\varphi_1+ \cdots +\varphi_j)(y) \neq 0$ for all $y \in K$
by construction,  we set  $h = \varphi_1+  \cdots +\varphi_j$  and conclude
that $g$ is an element of  $$\{g\in\stb (M) \, | \,\forall \,\text{ compact
}   K\subset   M   \  \exists\,h\in   \stb   (M)\:\text{s.t.}\:\zeri(h)\cap
K=\varnothing \: \text{and} \:hg\in\gta\}$$ as needed.

\end{proof}

\end{enumerate}

\end{remarks}

\begin{defn}\label{Loj}

An ideal  $\gta \subset \stb (M)$ is a {\em  \L  ojsiewicz  ideal} if

\begin{enumerate}

\item $\gta$ is generated by finitely many smooth functions $f_1,\ldots,f_l;$

\item $\gta$ contains an element $f$ with the property that for any compact $K\subset  M$ there exist a constant $c$ and an integer $m$ such that $|f(x)|\geq c\, d(x,  \zeri(\gta))^{m}$ on an open neighborhood of $K$.

\end{enumerate}

\end{defn}

\begin{remark} It is well known that in the definition above one can take $f$
to be the sum of squares of the generators $f_1^2 +\dots +f_l^2$. This can be seen as follows: $f_1, \ldots, f_l$ cannot be simultaneously flat at any point in $M$; otherwise, $f$ would be flat at some point of its zero set,
hence it could not verify the inequality in the definition. So $f_1^2 +\dots
+f_l^2$ is nowhere flat and dominates $C \,|f|^2$ on every compact set of $M$ for $C >0$ appropriately chosen. It thus verifies the required inequality with exponent $2m.$

%and by the classical \L  ojasiewicz inequality, it verifies the required inequality on every compact set of $M$.   

\end{remark}

\begin{lem}\label{L disug}

Let  $\gta$ be  a  \L  ojasiewicz ideal  generated  by $f_1,\ldots,f_l$  and
$f=f_1^2+\dots +f_l^2$.   Let $g\in \stb (M)$ be  such that $\zeri(g)\supset
\zeri(f) = \zeri(\gta)$. Then for  any compact set $K\subset M,$ there exist
a constant  $c$ and an  integer $m$  such that $g^{2m}  \leq cf$ on  an open
neighborhood  of $K$.  In  particular, there  exist  an integer  $m$ and  an
element $a\in \gta$  such that $g^{2m} \leq |a|$ on  an open neighborhood of
$K$.

\end{lem}

\begin{proof}

Let $X= \zeri(\gta)$. Assume $0\in X$.  Then for $x,y$ close to $0,$ one has
$|g(x)  - g(y)|  \leq c_1  \, d(x,y)$,  where $c_1$  is a  suitable positive
constant. Now $d(x,X) = \inf_{y\in  X}d(x,y),$ so since $g$ vanishes on $X,$
we obtain  $|g(x)| \leq  c_1\, d(x,X)$  on a neighborhood  of $0.$  Once the
compact set $K$ is fixed, we can find finitely many points $y_1,\ldots, y_s$
and for each  of them an open set $V_i$ where  the previous inequality holds
and  such that  $X\cap K\subset  V= \cup  V_i$. So  for a  suitable constant
$c_2,$  we  have $|g(x)|  \leq  c_2 \,  d(x,X)$  and  hence $g(x)^{2m}  \leq
c_2^{2m} d(x,X)^{2m}$ on $V$ for any $m.$ Also, if $K' \supset K$ is another
compact set such that $V\subset K',$ on a neighborhood of $K'$ one has $c \,
d(x,X)^{2m}  \leq  f(x)$.  Putting  together these  inequalities,  one  gets
$g^{2m} \leq c_3 f$ in $V,$ where $c_3 = \frac{c_2^{2m}}{c}.$

\end{proof}

%We conclude this section by globalizing the definition of the \L ojasiewicz radical.

The previous lemma prompts us to globalize the definition of the \L ojasiewicz
radical.

\begin{defn}\label{L radical}

\rm The {\em \L ojasiewicz radical}  of an ideal $\gta \in \stb(M)$ is given
by  $$\sqrt[\text{\L}]{\gta}:=\{g\in\stb   (M)\,  |  \,  \exists\,f\in\gta\:
\text{and} \: m\geq1\:\text{such that}\: f-g^{2m}\geq0\}.$$

\end{defn}

It is not hard to verify that $\sqrt[\text{\L}]{\gta}$ is a radical real ideal
for any ideal $\gta$.

%verify here that the argument for germs that this is an ideal and that is radical and real goes through.

\medskip

%\begin {center}

%*  *  *

%\end{center}

%\medskip

We can now prove our main result:

\begin{thm}\label{main}

Let $\gta \subset \stb (M)$ be a \L ojasiewicz ideal. Then

\begin{itemize}

\item    $\sqrt[\text{\L}]{\gta}$   is    analytic-like,    i.e.$   \widetilde
{\sqrt[\text{\L}]{\gta}}=\overline{\sqrt[\text{\L}]{\gta}}$

\item $\ideal(\zeri(\gta)) = \overline{\sqrt[\text{\L}]{\gta}}$.

\end{itemize}

\end{thm}

\begin{proof}

Note that for any ideal $\gtb$ we have

\begin{itemize}

\item   $\overline   \gtb\subset   \ideal(\zeri(\gtb))   $.    Indeed,   $g\in
\overline{\gtb}$ implies $T_x(g)  \in T_x(\gtb)$ for all $x\in  M$. Hence if
$x\in \zeri(\gtb),$ then $T_x(g)$ has  order $\geq 1$ because $T_x(\gtb)$ is
contained  in the  maximal  ideal of  the  ring of  formal  power series  at
$x$. Therefore, $g(x) =0.$

\item  $\zeri( \sqrt[\text{\L}]{\gtb})= \zeri(\gtb)$.

\end{itemize}

So       we       have$       \widetilde       {\sqrt[\text{\L}]{\gta}}\subset
\overline{\sqrt[\text{\L}]{\gta}}\subset  \ideal(\zeri(\gta))  $.  Hence  both
statements   will   be   proved   if  we   prove   $\ideal(\zeri(\gta))\subset
{\widetilde{\sqrt[\text{\L}]{\gta}}}$.

Take $g\in \ideal(\zeri(\gta)),$ and let  $f= f_1^2 =\dots+f_l^2$ be such that
$\zeri(f) =  \zeri(\gta)$. Let $K$  be a compact  set in $M$. By  Lemma \ref{L disug},  $g^{2m}\leq  cf$ on  a  neighborhood  of  $K$. Let  $\varphi_K  \in
\stb(M)$ be a  nonnegative function taking the value $1$ on  $K$ and the value
$0$  outside  the  neighborhood   where  the  inequality  above  holds.  Hence
$(\varphi_Kg)^{2m} \leq c f$ on the whole of $M$, which means $\varphi_K g \in
\sqrt[\text{\L}]{\gta}$. By Remark \ref{alike} (3), we are done.

\end{proof}

\medskip

%\begin {center}

%*  *  *

%\end{center}

\section{Weakly \L ojasiewicz ideals}\label{wloji}

\medskip

Looking at the proofs above, we see that the main ingredient was the existence
of a function  $f\in \gta$ that was  the sum of the squares  of the generators
and had the same zero-set as  $\gta,$ making $\gta$ \L ojasiewicz. In the next
lemma, we construct a function with  this property for a more general class of
ideals.

\begin{defn}\label{w L}

An ideal $\gta \subset \stb (M)$ is {\em weakly \L ojasiewicz} if

\begin{enumerate}

\item $\gta$  is locally finitely  generated, that is  for any $x\in  M$ there
exist finitely  many elements in  $\gta$ generating $\gta\,  \stb(U),$ where
$U$ is a suitable neighborhood of $x$;

\item There  exists an element $f\in  \tilde{\gta}$ such that  for any compact
$K\subset  M,$ there  exist a  constant $c$  and an  exponent $m$  such that
$|f(x)| \geq c \, d(x \zeri(\gta))^m$.

\end{enumerate}

\end{defn}

\begin{lem}\label{crespina}

Let  $\gta$  be  a weakly  \L  ojasiewicz  ideal.  Then there  exists  $f\in
\tilde{\gta}$  verifying property  (2) of  Definition \ref  {w L}  such that
$\zeri(f) = \zeri(\gta)$. Moreover, for  any compact set $K\subset M,$ there
exists a neighborhood $U$ of $K$ such that the restriction of $f$ to $U$ belongs
to $\gta \,\stb(U)$.

\end{lem}

\begin{proof}

Since  $\gta$ is locally  finitely generated  we can  assume it  is globally
generated by countably many  smooth functions $\{f_j\}_{j>0}$. Denote by $h$
the function  making $\gta$  a weakly \L  ojasiewicz ideal.  Since  $h_x \in
\gta  \stb_x$,  for   any  $x\in  M$  there  is  $l_x$   such  that  $h_x  =
\sum_{j=1}^{l_x}a_{jx} f_j$ and this  equality holds in a neighborhood $U_x$
of $x$.  Hence if  $K\subset M$ is  a compact  set, there exist finitely many
points  $x_1,  \ldots,  x_s$   such  that  $K\subset  U_{x_1}\cup\dots  \cup
U_{x_s}.$ Take  $\displaystyle l= \max_i\{l_{x_i}\},$ and  let $\{\varphi_i\}$ be  a smooth partition of  unity subordinated  to the covering  of $U  = U_{x_1}\cup\dots
 \cup     U_{x_s}$.    Then    $h     =    (\sum_i\varphi_i)h     =    \sum_i
\varphi_i(\sum_{j=1}^{l}a_{jx_i}                         f_j)=\sum_j^l(\sum_i
\varphi_ia_{jx_i})f_j$. This shows that  $h$ belongs to $\gta\, \stb(U)$ and
that the latter is a \L ojasiewicz ideal.

Next, take  an exhaustion of $M$  by compact sets  $\{K_j\}_{j>0}$ such that
$K_j\subset\Int {K_{j+1}}$ for  every $j \geq 1.$ We  can assume that $\gta$
is generated on a neighborhood  of $K_j$ by $f_i,\ldots , f_{i_j}$. Consider
the  open locally  finite covering  of $M$  given by  $\{ U_j=  \Int K_{j+1}
\setminus K_{j-2}\}_{j  \geq 1},$ where  $K_{-1}= K_0 = \emptyset.$  Let $\{
\alpha_j\}_{j \geq  1}$ be a  collection of smooth  functions $\alpha_j:M\to
[0,1]$ satisfying  that $\alpha_j  =1$ on $K_j  \setminus \Int  K_{j-1}$ and
$\supp  (\alpha_j) \subset  U_j$  for all  $j  \geq 1.$    Note
that  $ \alpha_j  f_1,  \ldots,  \alpha_j f_{i_{j}}$ still generate $\gta$ in a neighborhood $V_j\subset K_{j+1}$ of $K_j \setminus \Int K_{j-1}$ and that $\gta$ is \L ojasiewicz on $V_j$.

Now put

$$f = \sum_{j=1}^{\infty} \alpha_j \Big( \sum _{i=1}^{i_{j}}f_i^2\Big)$$

%=\Big(\sum_{j=1}^{\infty}\varepsilon _j \alpha_j\Big)\big(f_1^2+\ldots +f_{i_2}^2\big) + \Big(\sum_{j=2}^{\infty}\varepsilon _j \alpha_j\Big) \big(f_{i_2+1}^2+\ldots +f_{i_3}^2\big)+\ldots  $$

We get

\begin{enumerate}

\item $f\in \stb (M)$. Indeed, for any $x\in K_j\setminus \Int K_{j-1}
 \subset M, f$ is the sum of 3 summands, $f= \alpha_{j-1}\Big( \sum_{i=1}^{i_{j-1}}f_i^2\Big) + \alpha_j \Big( \sum _{i=1}^{i_{j}}f_i^2\Big) + \alpha_{j+1}\Big( \sum _{i=1}^{i_{j+1}}f_i^2\Big)$.

\item  $f \in  \tilde\gta$. Indeed,  for $x\in  K_j \setminus K_{j-1},$ the  germ $f_x$  belongs  to the  ideal generated  by $f_1,\ldots  ,  f_{i_{j+1}}$, which  generate the ideal $\gta$ on $V_{j+1}.$

\item $f\geq 0$ and $\zeri(f) = \zeri(\gta)$ since this is true locally.

\item  $f$  verifies  the  inequality  of Definition  \ref{w  L}.  Indeed,  if  $K\subset M$ is  a compact set, then $K\subset K_j$ for  some $j$. Hence $f$  belongs to the restriction of $\gta$ to $V_1 \cup \cdots \cup V_j,$ which is  a  \L  ojasiewicz  ideal, and  $f\big|_{V_1  \cup  \cdots  \cup V_j}$  is  a  combination with positive coefficients of the squares of its generators.

\end{enumerate}

\end{proof}

\begin{remark} A weakly \L  ojasiewicz  ideal is locally \L  ojasiewicz, and it is not hard to prove the converse.

\end{remark}

Next, note that Lemma \ref{L disug} and Theorem \ref{main} hold true for a weakly \L ojasiewicz ideal $\gta$ with the same proof, simply replacing $f_1^2+\ldots +f_l^2$ by the function $f$ constructed in Lemma \ref{crespina} above, provided the ideal $\gta$ is saturated as $f\in \tilde \gta$.

Hence we obtain:

\begin{thm}\label{weak L}

  Let $\gta \subset \stb (M)$ be  a saturated weakly \L ojasiewicz ideal. Then  $\ideal(\zeri(\gta  ))= \widetilde  {\sqrt[\text{\L}]{\gta}}$. If  $\gta$ is  not          saturated,          then         $\ideal(\zeri(\gta          ))  =\widetilde{\sqrt[\text{\L}]{\widetilde{\gta }}}$\end{thm}

\medskip

%\begin {center}

%*  *  *

%\end{center}

\section{Consequences}\label{conseq}

\medskip

\subsection{Answering a modification of the Bochnak conjecture}\label{modboc}

We now want to relate the notion of being \L ojasiewicz  with convexity.

We say that an ideal $\gta$ of  $\stb (M)$ is \em convex \em if each $g\in\stb(M)$  satisfying  $|g|\leq  f$  for  some $f\in\gta$  belongs  to  $\gta$.  In particular, the  \L ojasiewicz's radical $\sqrt[\text{\L}]{\gta}$  of an ideal $\gta$ of $\stb  (M)$ is a radical convex ideal. Moreover,  we define the {\em  convex hull} $\gtg(\gta)$ of an ideal $\gta$ of $\stb (M)$ by $$\gtg(\gta):=\{g\in\stb  (M)\, | \, \exists f\in\gta\:\text{such that}\:|g|\leq f\}.$$ Note  that  $\gtg(\gta)$ is  the  smallest convex  ideal  of  $\stb (M)$  that contains $\gta$ and $\sqrt[\text{\L}]{\gta}:=\sqrt{\gtg(\gta)}$.

Hence if  $\gta$ is convex  and radical, it  coincides with its  \L ojasiewicz
radical, and we immediately get:

\begin{cor}

If the  ideal $\gta\subset\stb (M)$ is  a (weakly) \L  ojasiewicz ideal, the
following are equivalent

\begin{enumerate}

\item $\gta$ has the zero property.

\item $\gta$ is closed, convex, and  radical.

\end{enumerate}

\end{cor}

%e already know ${\mathfrak C}_1(\sqrt[\text{\L}]{\gta})\subset{\mathfrak

%C}_2(\sqrt[\text{\L}]{\gta})\subset \overline{\sqrt[\text{\L}]{\gta}}\subset

%\ideal(\zeri(\gta))$. So it is enough to show that

%$\ideal(\zeri(\gta))\subset {\mathfrak C}_1(\sqrt[\text{\L}]{\gta})$.

%Bochnak's result doesn't have the Whitney closure, while AL is for potentially non finitely generated ideals. Either we succeed to remove the closure to get Bochnak's result (potentially through Malgrange's closure result, except the real radical can be infinitely generated even if the ideal is a priori finitely generated) or perhaps this argument implies the theorem in the case the ideal has countably many generators. To be decided later.

\medskip

%\begin{center}

%* * *

%\end{center}

%\medskip

\subsection{Recovering the Bochnak and Adkins-Leahy Nullstellensatz results}\label{balnss}

To compare our  results with the ones by Bochnak and  Adkins-Leahy, we have to relate  \L ojasiewicz  radicals  with  real radicals  of  ideals generated  by
analytic functions.  So assume $M$ is  an analytic manifold  and $\gta \subset
\stb (M)$ is generated by analytic  functions. It follows that the zero-set of
$\gta$  is  a  global  analytic   set  $X$  and  $\gta$  is  locally  finitely
generated. Furthermore,  there exists an analytic function  $f\in \tilde \gta$
whose  zero  set   is  $X,$  making  $\gta$  a   weakly  \L  ojasiewicz  ideal
(\cite{ABF}).

\bigskip

\begin{thm}\label{comp}

 Let $M$ be an analytic manifold and $\gta \in \an (M)$ be an ideal of real analytic functions. Then $$\widetilde{\sqrt[\text{\L}]{\widetilde{\gta \stb (M)}}} =\overline{\sqrt[\text{r}]{\gta \stb(M)}}$$.

 \end{thm}

\begin{proof}

Let $X= \zeri(\gta),$ and consider the ideal $\sqrt[\text{\L}]{\gta}\subset \an (M)$. We have the following:

\begin{enumerate}\parindent = 0pt

\item      [$\bullet$]      $     (\sqrt[\text{\L}]{\gta})(\an_x)      \subset
\sqrt[\text{\L}]{\gta\an_x}  =\sqrt[\text{r}]{\gta\an_x}.$ Indeed,  if $g\in
(\sqrt[\text{\L}]{\gta})(\an_x),$ then $g =\sum_i  h_i a_i,$ where $ h_i \in
\an_x$ and $ a_i \in \sqrt[\text{\L}]{\gta}$. Hence $a_i^{2m_i} \leq c_i f,$
for some $f\in \gta,$ and so $(h_ia_i)^{2m_i} \leq c'_i f,$ for $f\in \gta,$
which   means    $h_ia_i   \in   \sqrt[\text{\L}]{\gta\an_x}$.    So   $g\in
\sqrt[\text{\L}]{\gta\an_x}$. Since  the \L ojasiewicz  radical contains the
real radical, the last equality is the Risler Nullstellensatz in the ring of
germs of analytic functions; see \cite{rislerannss}.

\item               [$\bullet$]$              (\sqrt[\text{\L}]{\gta})(\stb_x)
\subset\sqrt[\text{\L}]{\gta\stb_x}$ by the same argument as above.

\item[$\bullet$]  $ \sqrt[\text{\L}]{\gta\stb_x}  = (\sqrt[\text{\L}]{\gta\stb
(M)}\,)_x$.  Indeed if  $\varphi \in  \gta \stb_x,$  then $\varphi  = \sum
\varphi_i a_i,$ for  $\varphi_i \in \stb_x$ and $ a_i  \in \gta$. This holds
true in an open neighborhood $U$  of $x$. Take a smooth bump function $\psi$
such that $\psi=1$ in a smaller neighborhood and its support is contained in
$U$. Then $\psi\varphi = \sum(\psi \varphi_i)a_i \in \gta\, \stb(M)$ and its
germ at $x$ is precisely $\varphi$.

\end{enumerate}

So    far    we     have    obtained

$$    (\sqrt[\text{\L}]{\gta})(\stb_x)
\subset\sqrt[\text{\L}]{\gta\stb_x}  = (\sqrt[\text{\L}]{\gta\stb (M)}\,)_x.$$

Now  apply  the Taylor  homomorphism  at  $x$.

$$ T_x  (\sqrt[\text{\L}]{\gta}
\stb_x)       =       (\sqrt[\text{\L}]{\gta})_x      \Fhaz_x\subset       T_x
(\sqrt[\text{\L}]{\gta\stb   (M)})_x  \subset\sqrt[\text{r}]{\gta\an_x}\Fhaz_x
.$$

It  is worth  noting  here that  we  identify the  Taylor  series of  analytic
functions with the  corresponding germs. The last inclusion  holds because the
elements of $(\sqrt[\text{\L}]{\gta\stb (M)})_x$  vanish on $X_x,$ hence their
Taylor       series       belong       to       $\ideal^{\Fhaz_x}(X_x)       =
\sqrt[\text{r}]{\gta\an_x}\,\Fhaz_x$  by   Malgrange's  theorem  \cite[Thm.3.5
p.90]{malgrange}. 
Arguing  as  in  \cite{al2},  $\sqrt[\text{r}]{\gta  \an_x}
\subset\sqrt[\text{r}]{\gta \stb_x}=(\sqrt[\text{r}]{\gta  \stb (M)}\,)_x$,
hence the last inclusion is an equality.

We have now finished  making all preliminary observations and  are ready to 
prove the theorem. 
Consider  $g \in \widetilde{\sqrt[\text{\L}]{\widetilde{\gta \stb
(M)}}}$. For any compact set $K\subset M,$ there is an open neighborhood
$U$  of $K$ such  that $g$  belongs to  $\sqrt[\text{\L}]{\widetilde{\gta \stb
(U)}}$.  It  follows  $g^{2m}  \leq  cf$  for $  f  \in  \widetilde  {\gta
\stb(U)}$. In turn, on a smaller neighborhood $V\subset U$ of $K,$ we get $f
\in  \gta\, \stb(V)$.  This means  that for  any  $x \in  M,$ the  germ $g_x$
belongs      to       $\sqrt[\text{\L}]{\gta\,\stb_x}$,      which      equals
$(\sqrt[\text{\L}]{\gta\,\stb (M)}\,)_x$. Applying  what was stated before, we
see    that   $T_x    g    \in   T_x(\sqrt[\text{\L}]{\gta\stb    (M)})\subset
(\sqrt[\text{r}]{\gta    \stb    (M)}\,)_x$,     which    implies    $g    \in
\overline{\sqrt[\text{r}]{\gta\stb  (M)}}$. The  reverse inclusion  comes from
the  fact that  the \L  ojasiewicz radical  is bigger  than the  real  one and
analytic-like by Theorem \ref{main}.

\end{proof}

As a consequence of Theorems \ref{main} and \ref{comp}, we recover the result of Adkins and Leahy in \cite{al2}. As for Bochnak's result, note that a finitely generated analytic ideal $\gta\subset \stb(M) $ is closed and if it is real, it coincides with its real radical.  Hence it has the zero property.

\thispagestyle{empty}

\bibliographystyle{amsalpha}

\bibliography{SmoothNSS}

\providecommand{\bysame}{\leavevmode\hbox to3em{\hrulefill}\thinspace}
\providecommand{\MR}{\relax\ifhmode\unskip\space\fi MR }
\providecommand{\MRhref}[2]{%
  \href{http://www.ams.org/mathscinet-getitem?mr=#1}{#2}
}
\providecommand{\href}[2]{#2}
\begin{thebibliography}{Now10}

\bibitem[ABF]{ABF}
Francesca Acquistapace, Fabrizio Broglia, and Jose~F. Fernando, \emph{On the
  {N}ullstellensatz for {S}tein spaces and real {C}-analytic sets},
  \textsl{{P}reprint.} arXiv:1207.0391v1, [math.CV] 2 Jul 2012.

\bibitem[AL75]{al1}
William~A. Adkins and J.~V. Leahy, \emph{Criteria for finite generation of
  ideals of differentiable functions}, Duke Math. J. \textbf{42} (1975), no.~4,
  707--716. \MR{MR0400287 (53 \#4122)}

\bibitem[AL76]{al2}
William~A. Adkins and John~V. Leahy, \emph{A {N}ullstellensatz for analytic
  ideals of differentiable functions}, Atti Accad. Naz. Lincei Rend. Cl. Sci.
  Fis. Mat. Natur. (8) \textbf{60} (1976), no.~2, 90--94. \MR{0460695 (57
  \#688)}

\bibitem[Boc73]{bochnak}
Jacek Bochnak, \emph{Sur le th\'eor\`eme des z\'eros de {H}ilbert
  ``diff\'erentiable''}, Topology \textbf{12} (1973), 417--424. \MR{MR0334271
  (48 \#12590)}

\bibitem[dB76]{debbumi}
Paolo de~Bartolomeis, \emph{Una nota sulla topologia delle algebre reali
  coerenti}, Boll. Un. Mat. Ital. (5) \textbf{13A} (1976), no.~1, 123--125.
  \MR{0425153 (54 \#13110)}

\bibitem[Koh79]{kohnacta}
J.~J. Kohn, \emph{Subellipticity of the {$\bar \partial $}-{N}eumann problem on
  pseudo-convex domains: sufficient conditions}, Acta Math. \textbf{142}
  (1979), no.~1-2, 79--122. \MR{MR512213 (80d:32020)}

\bibitem[Mal67]{malgrange}
B.~Malgrange, \emph{Ideals of differentiable functions}, Tata Institute of
  Fundamental Research Studies in Mathematics, No. 3, Tata Institute of
  Fundamental Research, Bombay, 1967. \MR{MR0212575 (35 \#3446)}

\bibitem[Now10]{nowak}
Krzysztof~Jan Nowak, \emph{On the real algebra of quasianalytic function
  germs}, \textsl{{P}reprint.}, November 2010.

\bibitem[Ris76a]{rislerannss}
Jean-Jacques Risler, \emph{Le th\'eor\`eme des z\'eros en g\'eom\'etries
  alg\'ebrique et analytique r\'eelles}, Bull. Soc. Math. France \textbf{104}
  (1976), no.~2, 113--127. \MR{0417167 (54 \#5226)}

\bibitem[Ris76b]{risler}
\bysame, \emph{Le th\'eor\`eme des z\'eros pour les id\'eaux de fonctions
  diff\'erentiables en dimension 2 et 3}, Ann. Inst. Fourier (Grenoble)
  \textbf{26} (1976), no.~3, x, 73--107. \MR{MR0425151 (54 \#13108)}

\bibitem[Tho67]{thom}
Ren{\'e} Thom, \emph{On some ideals of differentiable functions}, J. Math. Soc.
  Japan \textbf{19} (1967), 255--259. \MR{MR0211420 (35 \#2301)}

\bibitem[Tou72]{tougeron}
Jean-Claude Tougeron, \emph{Id\'eaux de fonctions diff\'erentiables},
  Springer-Verlag, Berlin, 1972, Ergebnisse der Mathematik und ihrer
  Grenzgebiete, Band 71. \MR{MR0440598 (55 \#13472)}

\end{thebibliography}

\thispagestyle{empty}

\end{document}